\documentclass[a4paper,11pt]{amsart}
\usepackage{amssymb,amsfonts, amsmath, amsthm}
\usepackage{todonotes}
\usepackage[mathscr]{euscript}
\usepackage{enumerate}
\usepackage{enumitem}
\usepackage{verbatim}
\usepackage{mathtools}
\usepackage{url}
\usepackage{tikz-cd}
\usepackage{xypic}
\usepackage{geometry}
\usepackage[cmtip,all]{xy}
\usepackage[pagebackref, colorlinks=true, linkcolor=blue, citecolor = red]{hyperref}
\usepackage{geometry}
\usepackage[nameinlink,capitalise,noabbrev]{cleveref}
\crefname{subsection}{Subsection}{Subsection}
\crefname{equation}{Diagram}{Diagram}

 \renewcommand*{\backref}[1]{}
 \renewcommand*{\backrefalt}[4]{({%
     \ifcase #1 Not cited.%
           \or On p.~#2%
           \else On pp.~#2%
     \fi%
     })}

\newtheorem{theorem}{Theorem}[section]
\newtheorem{lemma}[theorem]{Lemma}
\newtheorem{proposition}[theorem]{Proposition}
\newtheorem{corollary}[theorem]{Corollary}

\theoremstyle{definition}
\newtheorem{definition}[theorem]{Definition}

\theoremstyle{remark}
\newtheorem{remark}[theorem]{Remark}

\numberwithin{equation}{section}

\DeclareMathAlphabet{\mathbbe}{U}{bbold}{m}{n}

\def\DDelta{{\mathbbe{\Delta}}}
\newcommand{\DD}{\DDelta}

\newcommand{\C}{\mathscr{C}}
\newcommand{\s}{\mathscr{S}}
\newcommand{\set}{\mathscr{S}\mathrm{et}}
\newcommand{\cat}{\mathscr{C}\mathrm{at}}
\newcommand{\sset}{s\set}
\newcommand{\sS}{s\mathscr{S}}
\newcommand{\Tw}{\mathrm{Tw}}
\newcommand{\Hom}{\mathrm{Hom}}
\newcommand{\Map}{\mathrm{Map}}
\newcommand{\Fun}{\mathrm{Fun}}
\newcommand{\Ho}{\mathrm{Ho}}
\newcommand{\id}{\mathrm{id}}
\newcommand{\CoEq}{\mathscr{C}\mathrm{o}\mathscr{E}\mathrm{q}}

\title{Twisted Arrow Construction for Segal Spaces}
\date{January 2026}

\author{Chirantan Mukherjee}
\address{Ontario Research Centre for Computer Algebra, University of Western Ontario, London, Ontario, Canada}
\email{cmukher@uwo.ca}

\author{Nima Rasekh}
\address{Institut f\"ur Mathematik und Informatik, Universität Greifswald, Greifswald, Germany}
\email{nima.rasekh@uni-greifswald.de}


\begin{document}

\begin{abstract}
    We give an explicit description of the twisted arrow construction for simplicial spaces and demonstrate individually that it preserves the defining properties of a complete Segal space. Moreover, we show that for a Segal space, the natural projection from the twisted arrow Segal space is a left fibration.
\end{abstract}

\maketitle

\section{Introduction}

\subsection{Twisted arrow categories}
\emph{Twisted arrow categories} have proven extremely useful in the study of a variety of categorical settings, such as the study of the Yoneda embedding and computations of (lax) limits and (co)ends \cite{maclane1971categories,loregian2021coend}, categorical logic \cite{lawvere1970doctrines}, configuration spaces \cite{segal1973configuration}, algebraic $K$-theory \cite{waldhausen1985algebraick}, and the study of exponentiable fibrations \cite{johnstone1999exp,bungeniefield2000exp}.

More precisely, for a given category $\C$, the twisted arrow category is a functor of the form $\Tw\C \to \C^{op} \times \C$, which is in fact a discrete Grothendieck opfibration \cite{grothendieck1995descent}, corresponding to the Hom set functor: $\Hom(-,-)\colon\C^{op} \times \C \to \set$. It is precisely those fibrational properties and correspondence that is the foundation of the aforementioned applications of the twisted arrow construction.

Category theory has since been generalized to what is now called \emph{higher category theory}, $(\infty,1)$-category theory or simply $\infty$-category theory \cite{bergner2010surveyone}. Similar to the $1$-categorical situation, higher categorical analogues of the twisted arrow construction have also found a variety of applications, such as in derived geometry \cite{lurie2017ha,lurie2011dagx}, algebraic $K$-theory of higher categories \cite{barwick2017mackey,barwickglasmannardin2018dualizingfibrations,boors2018waldhausen}, computing lax $\infty$-limits \cite{gepnerhaugsengnikolaus2017laxlimits}, and decomposition spaces \cite{hackneykock2022culf}.

Given its importance, it is unsurprising that the twisted arrow construction has been considered in various models of higher categories, such as quasi-categories \cite{lurie2017ha}, $2$-Segal spaces \cite{boors2020twosegal} and internal $\infty$-categories \cite{martini2021yoneda}. However, while the 1-categorical twisted arrow construction is reasonably straightforward, the higher categorical versions tend to be more involved, often relying on advanced machinery specific to the model at hand. What is needed, then, is a straightforward, minimal, and broadly accessible approach to the $\infty$-categorical twisted arrow construction in a key model of higher categories. This is the aim of this work.

\subsection{Construction of twisted arrow complete Segal spaces}
In this work we focus on the twisted arrow construction in the context of complete Segal spaces, which is a prominent model of $(\infty,1)$-categories introduced by Rezk \cite{rezk2001css} and further studied in a variety of settings \cite{toen2005unicity,joyaltierney2007segalvqcat}. First, in \cref{thm:completeness} we give a short and self-contained proof that the twisted arrow construction preserves the three defining characteristics of a complete Segal space (Reedy fibrancy, Segal condition and completeness condition), only assuming basic aspects of the complete Segal space model structure as given in the original paper \cite{rezk2001css}. 

The definition of the complete Segal space twisted arrow construction employed is directly analogous to other constructions, such as \cite[Definition 4.2.4]{martini2021yoneda} in the context of internal $\infty$-categories. Moreover, some conditions can already be found in the literature, such as the Segal conditions (\cite[Theorem 2.11]{boors2020twosegal}, \cite[Proposition 4.2.5]{martini2021yoneda}) or the completeness condition (\cite[Proposition 4.2.5]{martini2021yoneda}). However, we are unaware of an explicit proof of the Reedy condition in the literature.

Beyond that, we show that the natural projection map from the twisted arrow Segal space is a left fibration (\cref{thm:left fib}). This shows that being a left fibration does not require the completeness condition. This result generalizes analogous results regarding the relation between completeness and over-categories in the context of Segal spaces \cite[Theorem 3.44]{rasekh2023yoneda}, \cite[Proposition 8.13]{riehlshulman2017rezkobject}.

While this result has been anticipated, no explicit statement can be found in the literature. Indeed, the model of quasi-categories does not allow an easy separation of the completeness condition, as it is built into the definition.\footnote{More explicitly, quasi-categories are directly equivalent to \emph{complete} Segal spaces \cite{joyaltierney2007segalvqcat}.} In the context of internal $\infty$-categories an analogous result can be extracted from the proof of \cite[Proposition 4.2.5]{martini2021yoneda}.

\subsection{Acknowledgment}
We would like to thank the anonymous referee for pointing to an error in a previous version. We would also like to thank the second referee for their suggestions, which greatly improved the exposition, and particularly an elegant suggestion on how to significantly improve the proof of \cref{lemma:twisted equiv}. The first author would like to thank the University of Trento for financial support in the form of a Thesis Research Abroad Scholarship. The second author is grateful to the Max Planck Institute for Mathematics in Bonn for its hospitality and financial support. 

\section{Reviewing Concepts}
This section establishes background concepts regarding the twisted arrow category (\cref{subsec:twisted}), simplicial sets (\cref{subsec:sset}), simplicial spaces (\cref{subsec:simplicial spaces}), and complete Segal spaces as a model of $(\infty,1)$-categories (\cref{subsec:css}). We will take standard categorical background for granted and refer the reader to \cite{maclane1971categories,riehl2017context} for further details. We will also assume basic aspects of model category theory. We refer the reader to \cite{hovey1999modelcategories,hirschhorn2003modelcategories} for foundational definitions and review here only the relevant aspects, such as the Kan model structure (\cref{subsec:Kan Model Structure}), the Reedy model structure (\cref{subsec:reedy model}), and the model structure for complete Segal spaces (\cref{subsec:complete Segal space Model Structure}).

\subsection{The Twisted Arrow Category} \label{subsec:twisted}
Let us quickly review the twisted arrow construction for ordinary categories. Given a category $\C$, let $\Tw \C$ be the category defined as follows:
\begin{itemize}[leftmargin=*]
    \item Objects are morphisms in $\C$
    \item A morphism from $f\colon x \to y$ to $g\colon z \to w$ is a commutative diagram of the form
    \[ 
    \begin{tikzcd}
    x \arrow[d, "f"'] & z \arrow[d, "g"] \arrow[l, "u"']\\
    y \arrow[r, "v"'] & w 
    \end{tikzcd} 
    \]
\end{itemize}
Composition is defined in the obvious way. There is a natural projection functor $\Tw \C \to \C^{op} \times \C$, the \emph{twisted arrow projection}, sending an object $f\colon x \to y$ in $\Tw\C$ to the object $(x,y)$ and a morphism given by the above commutative diagram to the morphism $(u,v)\colon (x,y) \to (z,w)$ in $\C^{op} \times \C$. This functor is a discrete Grothendieck opfibration and the corresponding functor $\C^{op} \times \C \to \set$ is the Hom set functor $\Hom_\C(-,-)$. For more details regarding the correspondence, see \cite[Theorem 2.1.2]{loregianriehl2020fibrations}.

\subsection{Simplicial Sets} \label{subsec:sset}
Let $\DD$ be the simplex category, with objects $[n] = \{0 < ... < n\}$ and morphisms order-preserving maps. The category of simplicial sets $\s$, also called \emph{spaces}, is defined as the functor category $\Fun(\DD^{op}, \set)$. We will use the following notation with regard to spaces:
 \begin{enumerate}[leftmargin=*]
  \item $\Delta[n]$ denotes the simplicial set representing $[n]$ i.e., $\Delta[n]_k = \Hom_{\DD}([k], [n])$. 
  \item $\partial \Delta[n]$ denotes the boundary of $\Delta[n]$ i.e., the largest sub-simplicial set which does not include $\mathrm{id}_{[n]}\colon [n] \to [n]$.
  \item $\Lambda[n]_i$ denotes the $i^{th}$ horn of $\Delta[n]$, i.e., the largest sub-simplicial set of $\Delta[n]$ which does not include the face map $d_i\colon [n-1] \to [n]$. 
 \end{enumerate}
 See \cite[Section I.1]{goerssjardine2009simplicialhomotopytheory} for a more detailed introduction.

\subsection{The Kan Model Structure} \label{subsec:Kan Model Structure}
The category of spaces has a Kan model structure, uniquely characterized as follows:
 \begin{itemize}
  \item[(F)] A map $f\colon Y \to X$ is a \emph{Kan fibration} if it has the right lifting property with respect to all horn inclusions $\Lambda[n]_i \to \Delta[n]$ for all $n \geq 1$ and $0 \leq i \leq n$.
  \item[(F$\cap$W)] A map $f\colon Y \to X$ is a \emph{trivial Kan fibration} if it has the right lifting property with respect to all boundary inclusions $\partial \Delta[n] \to \Delta[n]$ for all $n \geq 0$.  
  \item[(C)] A map $f\colon A \to B$ is a \emph{Kan cofibration} if it is a monomorphism.
 \end{itemize}
 In particular, the Kan model structure includes a notion of \emph{Kan equivalence}, which intuitively corresponds to the notion of \emph{weak homotopy equivalence} of spaces. Moreover, the fibrant objects are precisely the \emph{Kan complexes}. See \cite[Section I.11]{goerssjardine2009simplicialhomotopytheory} for further details.

\subsection{Simplicial Spaces} \label{subsec:simplicial spaces}
A \emph{simplicial space} is a functor $X \colon \DD^{\mathrm{op}} \to \sset$. Let $\sS = \Fun(\DD^{op}, \s) $ denote the category of simplicial spaces. We have the following basic notation with regard to simplicial spaces:
 \begin{enumerate}[leftmargin=*]
  \item For a given simplicial space $X$, we denote its value at $[n]$ by $X_n$, which is a space.
  \item We embed the category of spaces inside the category of simplicial spaces as constant simplicial spaces (i.e., the simplicial spaces $S$ such that $S_n = S_0$ for all $n$ and all simplicial operator maps are identities). 
  \item There is a functor $N\colon\cat \to \sS$, called the \emph{nerve} functor, defined by 
  \[(N\C)_{kn} = \Fun([k], \C).\]
  \item Let $F(n)$ be the simplicial space defined as $F(n)_{kl} = \Delta[n]_k = \Hom_{\Delta}([k],[n])$. Moreover, $\partial F(n)$ denotes the boundary of $F(n)$, meaning $\partial F(n)_{kl} = \partial \Delta[n]_k$. Alternatively, $F(n) = N[n]$.
  \item We will denote a morphism $[n] \to [m]$, and the corresponding morphism $F(n) \to F(m)$, by a sequence of numbers $\langle a_0, \dots, a_n \rangle$, where $a_i$ is the image of $i \in [n]$. 
  \item Denote $G(n)$ to be the sub-simplicial space of $F(n)$, with $G(n)_k \subseteq \Hom([k],[n])$ consisting of all $\varphi\colon [k] \to [n]$ such that $\varphi(k)-\varphi(0) \leq 1$. Following \cite[Section 5]{rezk2001css}, it is also described as a colimit
  \[G(n) \cong F(1) \coprod_{F(0)} ... \coprod_{F(0)} F(1). \]
  \item The category $\sS$ is enriched over spaces
  \[\Map_{\sS}(X,Y)_n = \Hom_{\sS}(X \times \Delta[n], Y).\]
  Here $\Delta[n]$ is the simplicial space given via the embedding defined above.
  \item By the Yoneda lemma, for a simplicial space $X$ we have a bijection of spaces
  \[X_n \cong \Map_{\sS}(F(n),X).\] 
  \end{enumerate}

 \subsection{Reedy Model Structure} \label{subsec:reedy model}
 The category of simplicial spaces has a Reedy model structure \cite{reedy1974modelstructure}, see also \cite[Section 2.4]{rezk2001css}, defined as follows:
 \begin{itemize}
  \item[(F)] A map $f\colon Y \to X$ is a (trivial) Reedy fibration if for each $n \geq 0$ the following map of spaces is a (trivial) Kan fibration
  $$ \Map_{\sS}(F(n),Y) \to \Map_{\sS}(\partial F(n),Y) \underset{\Map_{\sS}(\partial F(n), X)}{\times} \Map_{\sS}(F(n), X).$$
  \item[(W)] A map $f\colon Y \to X$ is a Reedy equivalence if it is a level-wise Kan equivalence.
  \item[(C)] A map $f\colon Y \to X$ is a Reedy cofibration if it is a monomorphism.
 \end{itemize}
 Unwinding definitions, an object $W$ is \emph{Reedy fibrant}, if for all $n \geq 0$, the map  
 \[\Map(F(n),W) \to \Map(\partial F(n),W) \]
 is a Kan fibration.

 Intuitively, the Reedy model structure models the homotopy theory of simplicial objects in spaces. Note that the Reedy model structure is \emph{simplicial}, meaning for every cofibration $i\colon A \to B$ of simplicial spaces and every fibration $p\colon Y \to X$ of simplicial spaces, the induced map of spaces
 \[ \Map_{\sS}(B,Y) \to \Map_{\sS}(A,Y) \underset{\Map_{\sS}(A,X)}{\times} \Map_{\sS}(B,X) \]
 is a fibration of spaces, which is a trivial fibration if either $i$ or $p$ is a weak equivalence. See \cite[Chapter 15]{hirschhorn2003modelcategories} for more details.

\subsection{(Complete) Segal Spaces} \label{subsec:css}
\emph{Complete Segal spaces} were introduced by Rezk as a model of $(\infty,1)$-categories \cite{rezk2001css}. See also \cite{rasekh2018introduction,mukherjee2022css} for further review. They are defined as simplicial spaces $W$ satisfying the following conditions:
 \begin{itemize}[leftmargin=*]
  \item \textbf{Reedy:} $W$ is Reedy fibrant.
  \item \textbf{Segal:} For each $n \geq 2$, the canonical map induced by the inclusions $G(n) \hookrightarrow F(n)$
  \[ W_n \cong \Map_{\sS}(F(n), W) \to \Map_{\sS}(G(n), W) \cong \underbrace{ W_1 \underset{{W_0}}{\times} \cdots \underset{{W_0}}{\times} W_1 }_{n-\;factors} \]
  is a Kan equivalence of spaces. By induction, this condition is equivalent to the map $W_n \xrightarrow{\simeq} W_{n-1} \times_{W_0} W_1$ being an equivalence, for all $n \geq 2$.
  \item \textbf{Complete:} The canonical map induced by the inclusion $F(0) \hookrightarrow NI$
  \[ \Map_{\sS}(NI, W) \to \Map_{\sS}(F(0), W) \cong W_0\]
  is a Kan equivalence of spaces, where $I$ is the category with two objects and a single isomorphism between them. 
\end{itemize}
If a simplicial space satisfies only the Reedy and Segal conditions, we call it a \emph{Segal space}. Already for a Segal space $W$, one can define a homotopy category $\Ho(W)$, with objects the points of $W_0$ and morphisms given by the path components of the mapping spaces defined via the pullback
\[ \Hom_{\Ho W}(x,y) = \pi_0(\{x\} \underset{W_0}{\times} W_1 \underset{W_0}{\times} \{y\}). \]
A morphism in the Segal space $W$ is a \emph{homotopy equivalence} if the corresponding class in $\Ho(W)$ is an isomorphism. We can hence define the subspace of \emph{homotopy equivalences} $W_{hoequiv} \subseteq W_1$, which has the property $W_{hoequiv} \simeq \Map_{\sS}(NI, W)$. This means the completeness condition can be rephrased as the requirement that the map
\[ W_{hoequiv} \to W_0 \]
is a Kan equivalence of spaces. See \cite[Section 5]{rezk2001css} for more details.

\subsection{(Complete) Segal Space Model Structure} \label{subsec:complete Segal space Model Structure}
We can modify the Reedy model structure on simplicial spaces to obtain two new model structures, the \emph{Segal space model structure} and the \emph{complete Segal space model structure}, which are simultaneously characterized as follows:
 \begin{itemize}[leftmargin=*]
  \item[(F)] A map $f\colon Y \to X$ between (complete) Segal spaces is a \emph{[trivial] (complete) Segal space fibration} if it is a [trivial] Reedy fibration.
  \item[(W)] A map $f\colon Y \to X$ is a \emph{(complete) Segal space equivalence} if for every (complete) Segal space $W$, the induced map of spaces
  \[ \Map_{\sS}(X, W) \to \Map_{\sS}(Y, W) \]
  is a Kan equivalence.
  \item[(C)] A map $f\colon Y \to X$ is a \emph{cofibration} if it is a monomorphism.
 \end{itemize}
 Note that the (complete) Segal space model structure is still simplicial. The complete Segal space model structure models the homotopy theory of $(\infty,1)$-categories, and is also known as the \emph{Rezk model structure}. See \cite[Section 7]{rezk2001css} for more details.

\subsection{Left Fibrations of Simplicial Spaces} \label{subsec:left fibrations}
Finally, we review \emph{left fibrations} of simplicial spaces. A map $p\colon L \to X$ of simplicial spaces is called a \emph{left fibration} if it satisfies two conditions:
\begin{itemize}
    \item \textbf{Reedy:} $p$ is a Reedy fibration.
    \item \textbf{Lifting:} for every $n \geq 1$, the following square is a homotopy pullback square of spaces:
    \[
    \begin{tikzcd}
    L_n \arrow[r, "\{0\}^*"'] \arrow[d, "p_n"'] & L_0 \arrow[d, "p_0"] \\
    X_n \arrow[r, "\{0\}^*"'] & X_0
    \end{tikzcd},
    \]
    where $\{0\}: [0] \to [n]$ is the unique map with image $0$ in $[n]$.
\end{itemize}
Intuitively, left fibrations are a higher categorical analogue of discrete Grothendieck opfibrations and model functors valued in spaces, which is made precise via the \emph{straightening construction}. See \cite{rasekh2023yoneda} for further details. Finally, via an elementary, but combinatorially involved argument, if $X$ is a Segal space, then the lifting condition reduces to the case $n =1$ \cite[Lemma $3.29$]{rasekh2023yoneda}.

\section{Twisted Arrow Simplicial Spaces}
In this section we define the twisted arrow construction for simplicial spaces $\Tw\colon \sS \to \sS$. Before we present a definition, we aim to provide a motivation by outlining a few key properties we would reasonably expect it to satisfy.

\begin{itemize}[leftmargin=*]
    \item \textbf{Lifting:} $\Tw$ lifts the twisted arrow category, meaning for a given category $\C$, $\Tw N\C = N\Tw\C$.
    \item \textbf{Simplicial Enrichment}: $\Tw$ is simplicially enriched, meaning for a simplicial set $S$ and a simplicial space $X$, $\Tw(X \times S) = \Tw(X) \times S$.
    \item \textbf{Colimit Preservation}: $\Tw$ preserves colimits.
\end{itemize}
Observe that, due to the first condition, we already know the functor $\Tw\colon \DD \to \sS$, is given by 
\[\Tw(F(n)) = \Tw(N[n])  = N \Tw([n]).\]
Following \cite[Theorem 4.51]{kelly2005enriched}, every simplicially enriched, colimit-preserving functor $\sS \to \sS$ is uniquely determined by its restriction to $\DD$, via left Kan extension. We thus obtain the following definition.

\begin{definition}
    Let $\Tw\colon \sS \to \sS$ be the following simplicially enriched left Kan extension
    \[
    \begin{tikzcd}
        \DD \arrow[d, "\textrm{Yon}"'] \arrow[r, hookrightarrow] &  \cat \arrow[r, "\Tw"] &  \cat \arrow[r, "N"] & \sS \\
        \sS \arrow[urrr, dashed, "\Tw"']
    \end{tikzcd} 
    \]
\end{definition}

Before we proceed with our main goal, namely, how this definition interacts with complete Segal spaces, we present a more explicit characterization, relying on the existing literature, such as \cite[Construction 5.2.1.1]{lurie2017ha} and \cite[Definition 4.2.4]{martini2021yoneda}.

For this next definition, let $\star$ denote the join of posets. In particular, recall that $[n] \star [m] \cong [n+m+1]$. 
\begin{definition}
    Let $Q\colon \DD \to \DD$ be the functor defined on objects by $Q([n]) = [n]^{op} \star [n]$, and on morphisms as $Q(\alpha) = \alpha^{op} \star \alpha$.
\end{definition}

We now have the following key observation. Recall here that $Q^*\colon \sS \to \sS$ is the precomposition functor defined by $Q$, using our definition $\sS = \Fun(\DD^{op}, \sset)$.

\begin{proposition}
    There is a natural isomorphism $\Tw \cong Q^*$.
\end{proposition}

\begin{proof}
 Following \cite[Theorem 4.51]{kelly2005enriched}, we only need to compare the two functors on the representable objects $F(n)$. However, by \cite[Construction 5.2.1.1]{lurie2017ha}, 
 $Q^*(F(n)) = N \Tw([n])$, which, by the first defining condition of $\Tw$ is precisely $\Tw(F(n))$. Hence, the result follows.
\end{proof}

\begin{remark}
    This result shows that our definition of $\Tw$ agrees with the definition of Martini \cite[Definition 4.2.4]{martini2021yoneda} in the context of internal $\infty$-categories.
\end{remark}

\section{Preservation Properties of the Twisted Arrow Construction}
We now show the twisted arrow construction preserves the defining properties of a complete Segal space.

\subsection{Reedy Fibrancy} \label{subsec:reedy}
We commence by showing that $\Tw$ preserves Reedy fibrancy. Before we proceed to the proof, we make the following observation regarding Reedy fibrancy of $\Tw W$. Let us define $\partial_{\Tw}F(2n+1)$ as the following coequalizer diagram
\begin{center}
    \begin{tikzcd}
    {\underset{0\le i< j\le n}{\coprod} F(2n-3)} \arrow[r, shift left] \arrow[r, swap, shift right] & {\underset{0\le i\le n}{\coprod}F(2n-1)}\arrow[r] & \partial_{\Tw} F(2n+1)
    \end{tikzcd}.
\end{center}
Using the fact that $\Tw$ preserves colimits and $\partial F(n)$ is precisely given via an analogous coequalizer diagram \cite[Proposition 2.3]{goerssjardine2009simplicialhomotopytheory}, it follows that $\Map(\partial F(n),\Tw W) \cong \Map(\partial_{\Tw} F(2n+1),W)$, which means in order to establish the Reedy fibrancy of $\Tw(W)$ we only need to show the map 
\[\Map(F(2n+1),W) \to \Map(\partial_{\Tw} F(2n+1),W)\]
is a Kan fibration.

We now proceed towards this goal by proving a technical lemma.

\begin{lemma} \label{lemma:sets}
     Let $\{A_i\}_{i \in I}$ be a collection of sets, and for each $i,j \in I$ let $f_{ij}, g_{ij}\colon A_{ij} \to A_i$ be injections. Assume there is a commutative diagram of sets 
    \begin{center}
        \begin{tikzcd}
        {\underset{i,j \in I}{\coprod} A_{ij}} \arrow[r, shift left] \arrow[r, swap, shift right] & {\underset{i \in I}{\coprod}A_i}\arrow[r] & Q
        \end{tikzcd}
    \end{center}
    and let $E$ be the coequalizer of the diagram. Then the induced map $E \to Q$ is injective if for all $i \in I$ the restriction $A_i \to Q$ is injective and for all $i \neq j$, $A_{ij} \cong A_i \times_Q A_j$.
\end{lemma}

\begin{proof}
Recall that the set $E$ is given by the quotient $\coprod_I A_i /\sim$, where $\sim$ is the equivalence relation generated by $f_{ij}(\gamma) \sim g_{ij}(\gamma)$ for all $\gamma \in A_{ij}$ and all $i,j \in I$. Let us assume two elements $\alpha,\beta$ in $E = \coprod_I A_i /\sim$ map to the same element in $Q$. If $\alpha,\beta \in A_i$, for some $i \in I$, then the desired result follows from the injectivity of the map $A_i \to Q$. Otherwise, by the pullback assumption, there exists a $\gamma \in A_{ij}$ such that $g_{ij}(\gamma) = \alpha$ and $f_{ij}(\gamma) = \beta$, meaning $\alpha \sim \beta$, which by definition means $\alpha = \beta$ in $E$. Hence, we are done.
\end{proof}
\begin{lemma} \label{lemma:reedy technical}
 The map $\partial_{\Tw} F(2n+1) \to F(2n+1)$ is a cofibration of simplicial spaces.
\end{lemma}

\begin{proof}
 By definition we need to show the map is injective. As $\partial_{\Tw} F(2n+1)$ is described as a coequalizer, we will apply \cref{lemma:sets}. First, the individual maps $F(2n-1) \to F(2n+1)$ are clearly injective. Next, by construction the intersection of the sets $F(2n-1)_k$ inside $F(2n+1)_k$ is given by $F(2n-3)_k$, meaning we have the following pullback square 
    \[
        \begin{tikzcd}
            F(2n-3)_k \arrow[r, "f"] \arrow[dr, phantom, "\lrcorner", very near start] \arrow[d, swap, "g"] & F(2n-1)_k \arrow[d] \\
            F(2n-1)_k \arrow[r] & F(2n+1)_k 
        \end{tikzcd}
    \] 
    The result now follows from \cref{lemma:sets}.
\end{proof}

With this lemma in hand, we can now establish the Reedy fibrancy of $\Tw W$. We will in fact prove a stronger result. 

\begin{proposition} \label{prop:reedy}
    Let $Y \to X$ be a Reedy fibration of simplicial spaces. Then the induced map $\Tw Y \to \Tw X$ is also a Reedy fibration.
\end{proposition}

\begin{proof}
    Following the definition of Reedy fibration (\cref{subsec:reedy model}) and the explanation above, we only need to establish that the map 
    \[ 
    \Map(F(2n+1),Y) \to \Map(\partial_{\Tw} F(2n+1),Y) \times_{\Map(\partial_{\Tw} F(2n+1),X)} \Map(F(2n+1),X)
    \]
    is a Kan fibration. However, $Y \to X$ is a Reedy fibration and the Reedy model structure is simplicial and so we only need to show that $\partial_{\Tw} F(2n+1) \to F(2n+1)$ is a cofibration, which follows from \cref{lemma:reedy technical}.
\end{proof}

Applying the previous result to the unique map $W \to F(0)$, we obtain the following corollary.

\begin{corollary} \label{cor:reedy}
    If $W$ is a Reedy fibrant simplicial space, then $\Tw W$ is also Reedy fibrant.
\end{corollary}

\subsection{Segal Condition}
Next we want to show that the twisted arrow construction preserves the Segal condition  (\cref{subsec:css}). 

\begin{proposition}\label{prop:segal}
If $W$ is a Segal space, then $\Tw W$ is also a Segal space.
\end{proposition}
\begin{proof}
By \cref{cor:reedy}, $\Tw W$ is a Reedy fibrant simplicial space. Hence, we just need to show that $\Tw W_{n} \to \Tw W_{n-1} \times_{\Tw W_0} \Tw W_1$ is an equivalence, which unwinds to showing that $W_{2n+1} \to W_{2n-1} \times_{W_1} W_3$ is an equivalence. Now, the following equivalence of cospans

    \[
        \begin{tikzcd}
            W_{2n-1} \arrow[d, equal] \arrow[r] & W_1 \arrow[d, equal] & W_3 \arrow[d, "\simeq"] \arrow[l, "\Tw(d_0)"'] \\
            W_{2n-1} \arrow[r] & W_1 & W_1 \times_{W_0} W_1 \times_{W_0} W_1 \arrow[l, "\pi_1"] 
        \end{tikzcd}
    \]
    induces via pullback the Kan equivalence 
    \[W_{2n-1} \times_{W_1} W_3 \to W_{2n-1} \times_{W_1} W_1 \times_{W_0} W_1 \times_{W_0} W_1 \cong W_{2n-1} \times_{W_0} W_1 \times_{W_0} W_1.\]
    This means in the following commutative triangle
    \[
        \begin{tikzcd}[column sep=0in]
            W_{2n+1} \arrow[dr, "\simeq"'] \arrow[rr] & & W_{2n-1}  \underset{W_1}{\times} W_3 \arrow[dl,"\simeq"] \\
        & W_{2n-1}\underset{W_0}{\times}W_1\underset{W_0}{\times}W_1
        \end{tikzcd}
    \]
    the right-hand map is a weak equivalence. Moreover, the left-hand diagonal map is an equivalence by the Segal condition and so the desired result follows by $2$-out-of-$3$.
\end{proof}    

\begin{remark}
   The result also follows as a special case of \cite[Theorem 2.11]{boors2020twosegal} as every Segal space is in particular a $2$-Segal space. Indeed, our argument can be seen as a specific instance of their more general approach using edgewise subdivision.
\end{remark}

\subsection{Completeness Condition}
In this subsection we show that $\Tw$ also preserves the completeness condition, relying on the terminology in \cref{subsec:css}. We start with a lemma characterizing the equivalences in the twisted arrow Segal space. 

\begin{lemma}\label{lemma:twisted equiv}
If $W$ is a Segal space, then a morphism $\sigma\colon f \to g$
\begin{center}
    \begin{tikzcd}
    \bullet \arrow[d, swap, "f"] & \arrow[l, swap, "k"] \bullet\arrow[d, "g"]\\
    \bullet \arrow[r, "h"] & \bullet
    \end{tikzcd}
\end{center}
in the twisted arrow Segal space $\Tw W$ is a homotopy equivalence if and only if $k$ and $h$ are homotopy equivalences in $W$.
\end{lemma}
\begin{proof}
If $\sigma$ is an equivalence then $h$ and $k$ are equivalences as the projection $\Tw W \to W^{op} \times W$ preserves equivalences. Conversely, assume $h$ and $k$ are equivalences. We want to prove that $\sigma$ is an equivalence. We will prove that $\sigma$ has a right inverse, as the case for the left inverse is analogous. Intuitively, we want to construct a map $\hat{\sigma}\colon F(5) \to W$ representing a diagram of the following form
\begin{center}
    \begin{tikzcd}
    \bullet \arrow[d, swap, "g"] & \arrow[l, "k^{-1}"] \bullet\arrow[d, "f"] & \bullet \arrow[d, "g"] \arrow[l, "k"] \arrow[ll, bend right=30, equal] \\
     \bullet \arrow[r, "h^{-1}"] \arrow[rr, bend right=30, equal] & \bullet \arrow[r, "h"] & \bullet  
    \end{tikzcd}.
\end{center}
By assumption, there is a map $\chi\colon F(2) \to W$, witnessing the fact that $hh^{-1}$ composes to the identity, and similarly a map $\kappa\colon F(2) \to W$ witnessing the fact that $k^{-1}k$ composes to the identity. 

Before we proceed, we need a technical construction. We define the simplicial space $\CoEq$ via the following pushout diagram
\[ 
\begin{tikzcd}[column sep=2cm]
G(3) \arrow[r, "\cong"]  \arrow[d, "\simeq"'] & \displaystyle F(1) \coprod_{F(0)} F(1) \coprod_{F(0)} F(1) \arrow[r, "{\langle 0,2 \rangle \coprod \langle 1,2 \rangle \coprod \langle 0,2 \rangle}"] & \displaystyle F(2) \coprod_{F(1)} F(3) \coprod_{F(1)} F(2) \arrow[d, "\simeq"] \arrow[ddr, bend left = 15, "{\langle 0,1,2 \rangle \coprod \langle 1,2,3,4 \rangle \coprod \langle 3,4,5 \rangle}" description]\\
F(3) \arrow[rr] \arrow[drrr, bend right = 10, "{\langle 0,2,3,5 \rangle}"'] & & \CoEq \arrow[ul, phantom, "\ulcorner", very near start] \arrow[dr, dashed]\\ 
& & & F(5) 
\end{tikzcd}.
\]
The left-hand map in the square is a Segal space equivalence by definition, hence the right-hand map in the square is one as well. Moreover, the map $G(5) \to F(5)$ factors through $i\colon G(5) \to F(2) \coprod_{F(1)} F(3) \coprod_{F(1)} F(2)$. This map $i$ is given as the colimit of the following diagram of Segal space equivalences
\[
 \begin{tikzcd}
 G(2) \arrow[d, "\simeq"] & F(1) \arrow[d, equal] \arrow[l] \arrow[r] & G(3) \arrow[d, "\simeq"] & F(1) \arrow[d, equal] \arrow[l] \arrow[r] & G(2) \arrow[d, "\simeq"]\\ 
 F(2) & F(1) \arrow[l] \arrow[r] & F(3) & F(1) \arrow[l] \arrow[r] & F(2) 
 \end{tikzcd} 
\]
and is hence a Segal space equivalence as well. Hence, we have the following commutative diagram.
\[
\begin{tikzcd}
 \displaystyle F(2) \coprod_{F(1)} F(3) \coprod_{F(1)} F(2) \arrow[d, "\simeq"] & G(5) \arrow[d, "\simeq"] \arrow[l, "\simeq", "i"']   \\ 
  \CoEq \arrow[r] & F(5) 
\end{tikzcd},
\]
hence, by $2$-out-of-$3$, the map $\CoEq \to F(5)$, which is the dashed map in the diagram above, is a Segal space equivalence.

We now complete the proof. Let $\id_g\colon F(3) \to W$ be the map representing the composition $(\id, g, \id)$. By assumption, the maps $\chi\colon F(2) \to W, \sigma\colon F(3) \to W, \kappa\colon F(2) \to W$ assemble into a map 
\[(\chi,\sigma,\kappa)\colon F(2) \coprod_{F(1)} F(3) \coprod_{F(1)} F(2) \to W \]
Moreover, the restriction via $\langle 0,2 \rangle \coprod \langle 1,2 \rangle \coprod \langle 0,2 \rangle$, is precisely $(\id, g, \id)\colon G(3) \to W$. Hence, by the universal property of the pushout, we have a map $(\chi, \sigma, \kappa, \id_g)\colon \CoEq \to W$. As $W$ is a Segal space, we can lift this to a map $\hat{\sigma}\colon F(5) \to W$, which by the lifting assumptions is a right inverse for $\sigma$, as it composes to the identity $\id_g$.
\end{proof}

\begin{proposition}\label{prop:twisted equiv pullback}
Let $W$ be a Segal space. Then we have the following strict and homotopy pullback square
\begin{center}
    \begin{tikzcd}
    \Tw(W)_{hoequiv}\arrow[dr, phantom, "\lrcorner", very near start]\arrow[r, hook]\arrow[d] & \Tw(W)_1\arrow[d]\\
    W_{hoequiv}^{op}\times W_{hoequiv}\arrow[r, hook] & W_1^{op}\times W_1
    \end{tikzcd}.
\end{center}
\end{proposition}
\begin{proof}
By definition, the right-hand map is isomorphic to the map 
\[\Map(F(3),W) \to \Map(F(1) \coprod F(1),W),\] 
which means it is a Reedy fibration. Hence, if we prove it is a strict pullback, it is also a homotopy pullback.

Now, by definition of pullbacks, the strict pullback is given by the collection of morphisms $\sigma\colon f \to g$ in $\Tw(W)$, such that their image $(k,h)$ in $W_1^{op} \times W_1$ lies in $W_{hoequiv}^{op} \times W_{hoequiv}$. By \cref{lemma:twisted equiv}, that is precisely $\Tw(W)_{hoequiv}$. Hence, we are done.
\end{proof}

\begin{lemma} \label{lemma:twisted equiv pullback}
    Let $W$ be a Segal space. The square 
    \begin{center}
        \begin{tikzcd}
        \Tw(W)_0\arrow[dr, phantom, "\lrcorner", very near start]\arrow[r]\arrow[d] & \Tw(W)_1\arrow[d]\\
        W_0^{op}\times W_0\arrow[r, "\simeq"]&W_1^{op}\times W_1
        \end{tikzcd}.
    \end{center}
    is a homotopy pullback square.
\end{lemma}

\begin{proof}
    We need to show the induced map 
    \[
    \Tw(W)_0 \to \Tw(W)_1 \times_{(W_1^{op}\times W_1)} (W_0^{op}\times W_0)
    \]
    is an equivalence, which unwinds to the map 
    \[
    \Map(F(1),W) \to \Map(F(3),W) \times_{\Map(F(1)\coprod F(1),W)} \Map(F(0) \coprod F(0),W).
    \]
    As $W$ is an arbitrary Segal space, it hence suffices to prove that the induced map 
    \[
         F(3) \coprod_{(F(1) \coprod F(1))} (F(0) \coprod F(0))  \to F(1)
    \] 
    is a Segal space equivalence. Precomposing the map with the equivalence $G(3) \to F(3)$ we get
    \[
    F(1) \cong G(3) \coprod_{(F(1) \coprod F(1))} (F(0) \coprod F(0)) \xrightarrow{ \ \simeq \ } F(3) \coprod_{(F(1) \coprod F(1))} (F(0) \coprod F(0)) \to F(1).
    \]
    The first map is a Segal space equivalence, as it is the result of a pushout of the following span of Segal space equivalences
    \[
    \begin{tikzcd}
    F(0) \coprod F(0) \arrow[d, equal] & F(1) \coprod F(1) \arrow[l] \arrow[d, equal] \arrow[r] & G(3) \arrow[d, "\simeq"] \\
    F(0) \coprod F(0) & F(1) \coprod F(1) \arrow[l] \arrow[r] & F(3). 
    \end{tikzcd}.
    \]
    The composition is the identity and hence also an equivalence. Hence, by $2$-out-of-$3$, the right-hand map is also an equivalence.
\end{proof}

We are finally ready to prove the main result of this section.
\begin{theorem}\label{thm:completeness}
If $W$ is a complete Segal space, then $\Tw W$ is a complete Segal space.
\end{theorem}
\begin{proof}
By \cref{prop:segal} we only need to check the completeness condition. We now have the following diagram. 
\begin{center}
    \begin{tikzcd}
    \Tw(W)_0\arrow[dr, phantom, "\lrcorner", very near start]\arrow[r]\arrow[d]&\Tw(W)_{hoequiv}\arrow[dr, phantom, "\lrcorner", very near start]\arrow[r, hook]\arrow[d]\arrow[d]&\Tw(W)_1\arrow[d]\\
    W_0^{op}\times W_0\arrow[r, "\simeq"]&W_{hoequiv}^{op}\times W_{hoequiv}\arrow[r, hook]&W_1^{op}\times W_1
    \end{tikzcd}.
\end{center}
By \cref{prop:twisted equiv pullback}, the right-hand square is a strict pullback square and a homotopy pullback square. Moreover, by \cref{lemma:twisted equiv pullback}, the whole rectangle is a homotopy pullback square. Hence, the left-hand square is a homotopy pullback square.

Since $W$ is a complete Segal space, $W_0^{op}\times W_0\xrightarrow{\simeq}W_{hoequiv}^{op}\times W_{hoequiv}$ is an equivalence of spaces, which implies that $\Tw(W)_0\xrightarrow{\simeq}\Tw(W)_{hoequiv}$ is an equivalence of spaces.
\end{proof}

\section{The Twisted Arrow Construction as a Left Fibration}
In this final section we demonstrate that if $W$ is a Segal space, then the twisted arrow projection $\Tw(W) \to W^{op} \times W$ is a left fibration (\cref{subsec:left fibrations}). First we need some technical lemmas.
\begin{lemma}\label{lemma:reedy fib}
If $W$ is a Segal space, then $\Tw W\to W^{op}\times W$ is a Reedy fibration.
\end{lemma}
\begin{proof}
We need to show that the map 
$$\Map(F(k),\Tw W) \to (W_k^{op}\times W_k)\underset{\Map(\partial F(k), W^{op}\times W)}{\times} \Map(\partial F(k), \Tw W)$$
is a Kan fibration.

By construction the left-hand side is given by $\Map(F(2k+1),W)$. Moreover, by the observation in the beginning of \cref{subsec:reedy}, we have $\Map(\partial F(k), \Tw W) \cong \Map(\partial_{\Tw} F(2k+1),W)$ and so the right-hand side simplifies to 
$$\Map(\partial_{\Tw} F(2k+1) \coprod_{(\partial F(k) \coprod \partial F(k))} (F(k) \coprod F(k)) ,W).$$

As $W$ is Reedy fibrant, it hence suffices to show that 
\[\partial_{\Tw} F(2k+1) \coprod_{(\partial F(k) \coprod \partial F(k))} (F(k) \coprod F(k)) \to F(2k+1)\]
is a cofibration or, in other words, a level-wise injection of sets. By \cref{lemma:sets}, this follows from the fact that the following commutative square 
\[
\begin{tikzcd}
 \partial F(k)_n \coprod \partial F(k)_n \arrow[dr, phantom, "\lrcorner", very near start] \arrow[r] \arrow[d] & \partial_{\Tw} F(2k+1)_n  \arrow[d]  \\ 
 F(k)_n \coprod F(k)_n \arrow[r] & F(2k+1)_n
\end{tikzcd}
\]
is a pullback, for every $n \geq 0$, which is immediate from the definition of the injections. Hence, we are done.
\end{proof}

\begin{lemma}\label{lemma:left fib one}
If $W$ is a Segal space, then the following diagram is a homotopy pullback square,
\begin{center}
    \begin{tikzcd}
    \Tw(W)_1\arrow[dr, phantom, "\lrcorner", very near start]\arrow[r]\arrow[d] & \Tw(W)_0\arrow[d]\\
    W_1^{op}\times W_1\arrow[r] & W_0^{op}\times W_0
    \end{tikzcd}.
\end{center}
\end{lemma}

\begin{proof}
By definition $\Tw(W)_0 = W_1$, and so the strict pullback is given by $\Tw(W)_1\simeq W_1\underset{W_0}{\times}W_1\underset{W_0}{\times}W_1$. The desired equivalence $\Tw(W)_1=W_3\simeq W_1\underset{W_0}{\times}W_1\underset{W_0}{\times}W_1$ now follows from the Segal condition.
\end{proof}

\begin{theorem}\label{thm:left fib}
If $W$ is a Segal space, then $\Tw W\to W^{op}\times W$ is a left fibration. 
\end{theorem}
\begin{proof}
By \cref{lemma:reedy fib}, the map is a Reedy fibration and so it suffices to check the appropriate homotopy pullbacks. As $W$ is a Segal space, by the explanation in \cref{subsec:left fibrations}, it is sufficient to check the homotopy pullback condition for $n=1$, which we have done in \cref{lemma:left fib one}.
\end{proof}

\begin{remark}
    In a sufficiently advanced theory of left fibrations, one can prove that a left fibration over a (complete) Segal space is itself a (complete) Segal space \cite[Lemma 3.38]{rasekh2023yoneda}. From this perspective, \cref{thm:left fib} could provide an alternative proof that the twisted arrow construction preserves complete Segal spaces. In fact, the approach of Martini in \cite[Proposition 4.2.5]{martini2021yoneda} precisely uses this proof strategy by first developing an advanced theory of left fibrations in the context of internal $\infty$-categories.
\end{remark}

\begin{remark} \label{rem:parametrized slice}
    As left fibrations are pullback stable, \cref{thm:left fib} implies that for a given object $x$ in a Segal space $W$, the pullback of $\Tw W \to W^{op} \times W$ along $\{x\} \times W \hookrightarrow W^{op} \times W$ is a left fibration $\Tw(W)_x \to W$, such that $(\Tw(W)_x)_0 = W_1 \underset{W_0}{\times} \Delta[0]$. This means $\Tw(W)_x \to W$ simply recovers the under-Segal space left fibration $W_{x/} \to W$ as described in \cite[Definition 3.41]{rasekh2023yoneda}, and so \cref{thm:left fib} is a direct generalization of \cite[Theorem 3.44]{rasekh2023yoneda}. 
\end{remark}

\bibliographystyle{alpha}
\bibliography{main.bib}

\end{document}